\documentclass[12pt]{amsart}
\usepackage{amssymb,amsmath,amsthm,latexsym}
\usepackage{mathrsfs}
\usepackage{a4wide}
\usepackage[usenames,dvipsnames]{color}
\usepackage{euscript}
\usepackage{graphicx}
\usepackage{comment}
\usepackage{mdwlist}
\usepackage{outlines}
\usepackage{mathtools,dsfont,wasysym}
\usepackage[makeroom,samesize]{cancel}

\usepackage[dvipsnames]{xcolor}
\usepackage{cite}
\usepackage{textgreek}
\usepackage{upgreek}
\usepackage{url}

\DeclareFontFamily{U}{mathx}{\hyphenchar\font45}
\DeclareFontShape{U}{mathx}{m}{n}{
      <5> <6> <7> <8> <9> <10>
      <10.95> <12> <14.4> <17.28> <20.74> <24.88>
      mathx10
      }{}

\newcommand{\nn}[1]{{\vert\kern-0.25ex\vert\kern-0.25ex\vert #1 
    \vert\kern-0.25ex\vert\kern-0.25ex\vert}}

\DeclareMathOperator{\supp}{supp}

\newtheorem{theorem}{Theorem}
\newtheorem{lemma}[theorem]{Lemma}

\theoremstyle{remark}

\newtheorem*{remark*}{Remark}
\theoremstyle{definition}
\newtheorem{definition}[theorem]{Definition}

\numberwithin{equation}{section}

\newcounter{maintheorem}

\newtheorem{mainth}[maintheorem]{Theorem}
\newtheorem*{mainthprime*}{Theorem A$^\prime$}
\makeatother

\renewcommand{\leq}{\leqslant}
\renewcommand{\geq}{\geqslant}
\newcounter{smallromans}

{\end{list}}

\newcommand{\sgn}{\text{sgn }}

\newcounter{smallromansdash}

{\end{list}}

\newcounter{bigromans} 
  {\end{list}}

\begin{document}

\baselineskip=17pt

\title{On isometries and Tingley's problem for the spaces $T\left[\theta, \mathcal{S}_{\alpha}\right], 1 \leq \alpha<\omega_{1}$}
\author{Natalia Ma\'slany}
\address{ Jagiellonian University, Doctoral School of Exact and Natural Sciences, Faculty of Mathematics and Computer Science, Institute of Mathematics, {\L}ojasiewicza 6, 30-348 Krak\'ow, Poland and Institute of Mathematics, Czech Academy of Sciences, \v{Z}itn\'{a} 25, 115~67 Prague 1, Czech Republic}
\email{nataliamaslany97@gmail.com}
\thanks{The author was supported by GAČR grant GF20-22230L; RVO 67985840 and received an incentive scholarship from the funds of the program Excellence Initiative - Research University at the Jagiellonian University in Krak\'{o}w.}

\begin{abstract}
     We extend the existing results on surjective isometries of unit spheres in the Tsirelson space $T\left[\frac{1}{2}, \mathcal{S}_1\right]$ to the class $T[\theta, \mathcal{S}_{\alpha}]$ for any integer $\theta^{-1} \geq 2$ and $1 \leqslant \alpha < \omega_1$, where $\mathcal{S}_{\alpha}$ denotes the Schreier family of order $\alpha$. This positively answers Tingley's problem for these spaces, which asks whether every surjective isometry between unit spheres can be extended to a surjective linear isometry of the entire space.

    Furthermore, we improve the result stating that every linear isometry on $T[\theta, \mathcal{S}_1]$ \linebreak\big($\theta \in \left(0, \frac{1}{2}\right]$\big) is determined by a permutation of the first $\lceil \theta^{-1} \rceil$ elements of the canonical unit basis, followed by a possible sign change of the corresponding coordinates and \linebreak a sign change of the remaining coordinates. Specifically, we prove that only the first $\lfloor \theta^{-1} \rfloor$ elements can be permuted. This finding enables us to establish a sufficient condition for being a linear isometry in these spaces.
\end{abstract}


\maketitle
\section{Introduction and the main result}

In 1987, Tingley \cite{tingley1987isometries} proposed a question that has since become known as Tingley's problem:

\emph{
Let $X$ and $Y$ be normed spaces with unit spheres $\mathbb{S}_{X}$ and $\mathbb{S}_{Y}$, respectively. Suppose that $U\colon \mathbb{S}_{X} \rightarrow \mathbb{S}_{Y}$ is a surjective isometry. Is there a linear isometry $\tilde{U}\colon X \rightarrow Y$ such that $\tilde{U} \vert_{\mathbb{S}_{X}} = U$?}

Many authors have shown that Tingley's problem has a positive solution for surjective isometries of unit spheres in classical Banach spaces $\ell_p(\Gamma)$, $L_p(\mu)$ $(1 \leq p \leq \infty)$, and $C(\Omega)$ (see, \emph{e.g.}, \cite{ding20021, ding2003isometric, ding2004representation, dding2004representation, ding2007isometric, ding2008isometric, fang2006extension, liu2007extension, liu2009extension, tan2009nonexpansive, tan2012extension, jian2004extension, yang2006extension}). However, the general case remains open. Notable results in the search for a solution to Tingley's problem in specific spaces have been comprehensively documented in surveys by A. M. Peralta \cite{peralta2018survey}, G. G. Ding \cite{ding2009isometric}, X. Yang, and X. Zhao \cite{yang2014extension}. Recently, a positive solution to this isometric expansion problem has been found for 2-dimensional Banach spaces (see \cite{banakh2022every}); nevertheless, the answer remains unknown for higher dimensions. Positive solutions for certain subspaces of function algebras, including closed function algebras on locally compact Hausdorff spaces, have been presented in more recent studies (see \cite{cueto2022exploring}).

The Tsirelson space $T$ (the dual of the space constructed by Tsirelson \cite{tsirelson1974impossible}, which was the first example of a space containing no isomorphic copies of $c_0$ or $\ell_p$ for $1\leq p < \infty$) can be regarded as a special case of the double-parameter family of Banach spaces $T[\theta, \mathcal{S}_{\alpha}]$, where $\theta \in \left(0, \frac{1}{2}\right]$ and $1 \leqslant \alpha < \omega_1$, with $\mathcal{S}_{\alpha}$ being the Schreier family of order $\alpha$, where $\alpha$ is a countable ordinal.

In \cite{maslany2023isometries} we have characterized linear isometries of combinatorial Tsirelson spaces. However, the methods employed assume linearity of the isometries throughout the entire space. We improve the main theorem from this article by proving the following first main result:

\begin{mainth}\label{Th:A}
    Let $\theta \in \big(0,\frac{1}{2} \big]$. Then $U\colon T\big[\theta, \mathcal{S}_1\big]\to T\big[\theta, \mathcal{S}_1\big]$ is a linear isometry if and only if
    \[
        U e_i = \left\{\begin{array}{ll} \varepsilon_i e_{\pi(i)}, &  1\leq i\leq \lfloor \theta^{-1}  \rfloor\\
        \varepsilon_i e_i, & i > \lfloor \theta^{-1}  \rfloor
        \end{array} \right. \quad(i\in \mathbb N)
    \]
    for some $\{-1,1\}$-valued sequence $(\varepsilon_i)_{i=1}^\infty$ and a permutation $\pi$ of $\big\{1,2,\ldots, \lfloor \theta^{-1}  \rfloor\big\}.$
\end{mainth}

Then, following the approach of \cite{tan2012isometries}, where the author determine the surjective isometries of the unit spheres of Tsirelson space $T[\frac{1}{2}, \mathcal{S}_1]$ and the modified Tsirelson space $T_M$ and answer Tingley’s problem affirmatively in these spaces, we establish the subsequent main Theorem.

\begin{mainth}\label{Th:B}
    Let $\theta^{-1} \geq 2$ be an integer and let $U\colon \mathbb{S}_{T[\theta, \mathcal{S}_{\alpha}]} \to \mathbb{S}_{T[\theta, \mathcal{S}_{\alpha}]}$ be surjective isometry. 
     If $\alpha = 1$, then 
   \[
    U \bigg( \sum_{i=1}^{\infty} a_i e_i \bigg) = \sum_{i=1}^{ \theta^{-1}  } \varepsilon_i a_i e_{\pi(i)} + \sum_{i= \theta^{-1} + 1}^{\infty} \varepsilon_i a_i e_i
    \]
    and if  $1< \alpha < \omega_1$, then 
    \[
         U \bigg(\sum_{i=1}^{\infty} a_i e_i\bigg) = 
        \sum_{i=1}^{\infty} \varepsilon_i a_i e_i,
    \]
    for every $\sum_{i=1}^{\infty} a_i e_i \in \mathbb{S}_{T[\theta, \mathcal{S}_{\alpha}]}$, where $(\varepsilon_i)_{i=1}^\infty$ is a $\{-1,1\}$-valued sequence and $\pi$ is a~permutation of $\big\{1,2,\ldots,  \theta^{-1} \big\}.$
\end{mainth}

This result together with Theorem \ref{Th:A} get an affirmative answer to the Tingley’s problem in combinatorial Tsirelson spaces $T\big[\theta, \mathcal{S}_{\alpha}\big]$ for an integer $\theta^{-1} \geq 2$ and $1\leqslant \alpha < \omega_1$.

\section{Preliminaries}

\subsection{Combinatorial spaces}
Let us denote by $(e_i)_{i=1}^\infty$ the standard unit vector basis of $c_{00}$ and by $[\mathbb{N}]^{<\omega}$ the family of finite subsets of $\mathbb{N}.$ We adopt the following notation for sets $F_1, F_2 \in [\mathbb{N}]^{<\omega}$: $F_1<F_2$ means that $\max F_1 < \min F_2$, and we say that $F_1$ and $F_2$ are \textit{consecutive} in this case. Additionally, we use the notation $F_1<n$ instead of $F_1 < \{n\}$ for $n\in \mathbb{N}$.

\begin{definition}\label{sp}
A family $\mathcal{F}\subset [\mathbb{N}]^{<\omega}$ is {\em regular}, whenever it is simultaneously
\begin{itemize}
		\item \textit{hereditary} \big($F \in \mathcal{F}$ and $G \subset F \implies G \in \mathcal{F}$\big);
		\item \textit{spreading} \big($\{l_1,l_2,\dots,l_n\} \in \mathcal{F}$ and $l_i \leqslant k_i \implies \{k_1,k_2,\dots,k_n\} \in \mathcal{F}$\big);
		\item \textit{compact} as a subset of the Cantor set $\{0,1\}^{\mathbb{N}}$ via the natural identification of $F \in \mathcal{F}$ with 
        \[
            \chi_F=\sum_{i \in F} e_i \in \{0,1\}^{\mathbb{N}}.
        \]
\end{itemize}
\end{definition}

The simplest examples of regular families include
\[
    \mathcal{A}_n := \big\{F\in[\mathbb{N}]^{<\omega} \colon |F|\leqslant n\big\} \quad(n\in \mathbb N)
\]

\emph{i.e.}, for a given $n\in \mathbb N$, the family of subsets of $\mathbb{N}$ having at most $n$ elements. The family of Schreier sets, introduced by Alspach and Argyros \cite{alspach1992complexity}, is defined using these families in the following manner.
\begin{definition}\label{d3}
 Given a countable ordinal $\alpha$, we define inductively the Schreier family of order $\alpha$  as follows:
 \begin{itemize}
     \item $\mathcal{S}_{0} := \mathcal A_1;$
     \item if $\alpha$ is a successor ordinal, \emph{i.e.}, $\alpha=\beta +1$ for some $\beta < \omega_1$, then 
        \[
            \mathcal{S}_{\alpha}:=\Bigg\{ \bigcup_{i=1}^{d} S^{i}_{\beta}\colon d \leq S^{1}_{\beta} < S^{2}_{\beta} < \cdots < S^{d}_{\beta},\,\, \big\{ S^{i}_{\beta} \big\}_{i=1}^d \subset \mathcal{S}_{\beta} \text{ and } d \in \mathbb{N} \Bigg\} \cup\big\{\emptyset\big\};
        \]
     \item if $\alpha$ is a non-zero limit ordinal and $\left(\alpha_{n}\right)_{n=1}^\infty$ is a fixed strictly increasing 
     sequence of successor ordinals converging to $\alpha$ with 
     $\mathcal S_{\beta_n} \subset \mathcal S_{\beta_{n+1}}$
     for all $n \in \mathbb{N}$, where $\alpha_n = \beta_n+1$ for all $n \in \mathbb{N}$, we set 
        \begin{equation*}
        \mathcal{S}_{\alpha} :=\big\{S_{\alpha_n} \in[\mathbb{N}]^{<\omega}\colon S_{\alpha_n} \in    \mathcal{S}_{\alpha_{n}},\,\, n \leq \min S_{\alpha_n} \text{ for some n} \big\} \cup\big\{\emptyset\big\}. 
        \end{equation*}
 \end{itemize}
\end{definition}

We emphasize that in the case where $\alpha$ is a limit ordinal, we require the sequence $(\alpha_n)_{n=1}^\infty$ cofinal in $\alpha$ to comprise successor ordinals as needed in the proof of Theorem \ref{Th:B}. We can assume, and we will, that $S_{\alpha_n} \subset S_{\alpha_{n+1}}$ for all $n \in \mathbb{N}$, which will also be employed in the proof of Theorem \ref{Th:B}. Indeed, repeating the proof of \cite[Proposition 3.2.]{causey2017concerning} in the case of Schreier families $\big\{ S_{\xi} \big\}_{\xi < \omega_1}$ which are multiplicative in the sense of \cite{causey2017concerning} we can also derive the desired result for Schreier families.
Elements belonging to $\mathcal S_{\alpha}$ are called $\mathcal S_{\alpha}$\emph{-sets}.
The fact that these families are regular is well-established; see \cite{causey2017concerning}[Proposition 3.2] or \cite{todorcevic2005ramsey}.

\subsection{Combinatorial Tsirelson spaces}
For a regular family $\mathcal{F}$ and $\theta\in \big(0,\frac{1}{2}\big]$, we define the Banach space $T[\theta,\mathcal{F}]$ specializing it later to space $T[\theta, \mathcal{S}_{\alpha}]$ for some countable ordinal $\alpha$.\smallskip

For a vector $x=(a_1, a_2, \ldots, a_n)\in c_{00}$ and a finite set $E\subset \mathbb N$, we employ the symbol $Ex$ to represent the projection of $x$ onto the space $[e_i \colon i \in E],$ given by
\begin{equation}
    E\Bigg(\sum_{i=1}^{n} a_i e_i\Bigg) = \sum_{i\in E} a_i e_i.
\end{equation}

We denote by $\|\cdot\|_0$ the supremum norm on $c_{00}.$ Suppose that for some $n \in \mathbb{N}$ the norm $\|\cdot\|_n$ has been defined. Let
\begin{equation*}\label{nnorm} 
 \|x\|_{n+1}= \max \big\{ \|x\|_n, \|x\|_{T_n} \big\}\quad (n\in \mathbb N),
\end{equation*}
where
\begin{equation*}
 \|x\|_{T_n}= \sup \bigg\{ \theta \sum_{i=1}^d \big\|E_ix\big\|_n\colon E_1 <\cdots < E_d,\, d \in \mathbb{N},\, \{E_i\}_{i=1}^d \subset [\mathbb{N}]^{<\omega}, \, \{\min E_{i}\}_{i=1}^d \in \mathcal{F}\bigg\}.  
\end{equation*}
We define the norm $\|x\|_{\theta,\mathcal{F}} := \sup_{n \in \mathbb{N}}\|x\|_n$ and denote by $T[\theta,\mathcal{F}]$ the completion of $c_{00}$ with respect to it. 

A proof by induction demonstrates that this norm is bounded above by the $\ell_1$-norm and is given by the following implicit formula for  $x \in T[\theta,\mathcal{F}]$: 
\begin{equation}\label{im}
    \|x\|_{\theta,\mathcal{F}}= \max \big\{\|x\|_{\infty} , \|x\|_T \big\},
\end{equation}
where
\begin{equation*}
    \|x\|_T= \sup\bigg\{\theta \sum_{i=1}^d \big\|E_ix\big\|_{\theta,\mathcal{F}} : E_1 < \cdots < E_d,\, d \in \mathbb{N},\, \{E_i\}_{i=1}^d \subset [\mathbb{N}]^{<\omega}, \,\{\min E_i\}_{i=1}^d \in \mathcal{F}\bigg\}.
\end{equation*}

It can be readily deduced from the definition that the unit vectors $(e_i)_{i=1}^\infty$ form an 1-unconditional basis of the space  $T[\theta,\mathcal{S}_{\alpha}]$ for a countable ordinal $\alpha$.

For $x_1,x_2 \in c_{00}$, we write $x_1<x_2$ whenever $\supp\, x_1 < \supp\, x_2$ and for $n\in\mathbb{N}$ we streamline the notation of $\supp\, x_1<n$ to $x_1<n.$

We adopt the following convention in this paper: we say that the norm of an element $x \in T[\theta,\mathcal{F}]$ is given by sets $E_1<E_2<\cdots < E_{d}$ for some $d\in \mathbb{N}$ (with $\{\min E_i\}_{i=1}^d \in \mathcal{F}$) precisely when
\begin{equation*}
    \|x\|_{\theta,\mathcal{F}}= \theta \cdot \sum_{i=1}^{d}\big\|E_i x\big\|_{\theta,\mathcal{F}}.
\end{equation*}

\section{Linear isometries on $T[\theta, \mathcal{S}_1]$ spaces for $\theta \in \big(0, \frac{1}{2}\big]$ }

For further considerations, let us fix $\theta \in \big(0,\frac{1}{2} \big]$ and let $\lfloor \theta^{-1} \rfloor$ and $\lceil \theta^{-1} \rceil$ be the floor and the ceiling of $\theta^{-1}$, respectively. Note that we do not yet require $\theta^{-1}$ to be an integer. Fix a~countable ordinal $\alpha \geq 1$. Throughout this paper we use $\mathbb{\mathbb{S}_{T[\theta, \mathcal{S}_{\alpha}]}}$to denote the unit sphere of $T[\theta, \mathcal{S}_{\alpha}]$.

In \cite[Theorem A]{maslany2023isometries} we have obtained the following description of linear isometries on combinatorial Tsirelson spaces.
\begin{theorem}\label{tst}
    Let $\theta \in \big(0,\frac{1}{2} \big]$. If $U\colon T\big[\theta, \mathcal{S}_1\big]\to T\big[\theta, \mathcal{S}_1\big]$ is a linear isometry, then
    \[
        U e_i = \left\{\begin{array}{ll} \varepsilon_i e_{\pi(i)}, &  1\leq i\leq \lceil \theta^{-1}  \rceil\\
        \varepsilon_i e_i, & i > \lceil \theta^{-1}  \rceil
        \end{array} \right. \quad(i\in \mathbb N)
    \]
    for some $\{-1,1\}$-valued sequence $(\varepsilon_i)_{i=1}^\infty$ and a permutation $\pi$ of $\big\{1,2,\ldots, \lceil \theta^{-1}  \rceil\big\}.$
\end{theorem}

  Armed with this result, we are now ready to prove Theorem~\ref{Th:A}.
  \begin{proof}
       Suppose that $U\colon T\big[\theta, \mathcal{S}_1\big]\to T\big[\theta, \mathcal{S}_1\big]$ is a linear isometry.If $\theta^{-1}$ is an integer, there is nothing to prove, so assume that this is not the case. Then $\lceil \theta^{-1}  \rceil - 1 = \lfloor \theta^{-1}  \rfloor$ and $\lceil \theta^{-1}  \rceil > \theta^{-1}$.
       
       It is enough to show that for any $i \neq \lceil \theta^{-1}  \rceil$ holds $\pi(i) \neq \lceil \theta^{-1}  \rceil$. 

       Let $i \in \{1,2,\ldots, \lceil \theta^{-1}  \rceil - 1\}$ and suppose for the contrary that $\pi(i)= \lceil \theta^{-1}  \rceil$, \emph{i.e.}, $Ue_i = \varepsilon_i e_{\lceil \theta^{-1}  \rceil}$. Then, by Theorem \ref{tst}, for any indices $\lceil \theta^{-1}  \rceil<j_1<j_2<\cdots <j_{\lceil \theta^{-1}  \rceil-1}$ we have 
       \[
         \Bigg\|Ue_i - \sum_{k=1}^{\lceil \theta^{-1}  \rceil - 1} Ue_{j_k}  \Bigg\| = \bigg\| \varepsilon_i e_{\lceil \theta^{-1}  \rceil} -  \sum_{k=1}^{\lceil \theta^{-1}  \rceil - 1} \varepsilon_{j_k} e_{j_k}  \bigg\| = \theta \cdot \lceil \theta^{-1}  \rceil > 1.
       \] 
       On the other hand, since $U$ is a linear isometry, we obtain
       \[
        \Bigg\|Ue_i - \sum_{k=1}^{\lceil \theta^{-1}  \rceil - 1} Ue_{j_k}  \Bigg\| =\Bigg\|Ue_i - U \Bigg(\sum_{k=1}^{\lceil \theta^{-1}  \rceil - 1} e_{j_k} \Bigg) \Bigg\| = \Bigg\|e_i - \sum_{k=1}^{\lceil \theta^{-1}  \rceil - 1} e_{j_k}\Bigg\| = 1.
       \] 
       This contradiction finishes the proof that the isometry has the desired form. 

        To show that all such maps are indeed isometries, we refer the reader to \cite[Theorem 4.1]{antunes2022surjective} for technical details, because the same argument as for $\theta^{-1}$ being an integer applies.
  \end{proof}

\section{Isometries on $\mathbb{S}_{T[\theta, \mathcal{S}_{\alpha}]}$ for an integer $\theta^{-1} \geq 2$ and $1\leqslant \alpha < \omega_1$}

To prove Theorem~\ref{Th:B} we need a series of lemmas; the proofs emulate that of \cite{tan2012isometries}.

\begin{lemma}\label{l3}
    Let $u,v \in \mathbb{S}_{T[\theta, \mathcal{S}_{\alpha}]}$.  Then for $\alpha = 1$ we have
    \begin{outline}[enumerate]
        \1 $\min ( \|u+y\|, \|u-y\| ) \leq 1$ for all $y \in \mathbb{S}_{T[\theta, \mathcal{S}_{\alpha}]}$ if and only if $u \in \{\pm e_1, \pm e_2, \ldots, \pm e_{\lfloor \theta^{-1} \rfloor}\}$;
        \1 If $v \geq \lfloor \theta^{-1}  \rfloor + 1$ and $\min ( \|v+y\|, \|v-y\| ) \leq \theta \cdot (\lfloor \theta^{-1}  \rfloor+1)$ for all $y \in \mathbb{S}_{T[\theta, \mathcal{S}_{\alpha}]}$, then $v$ has one of the following forms:
            \2  $|v_{\lfloor \theta^{-1} \rfloor+1}| = 1$ with $|v_i| \leq \theta$ for all $i\neq \lfloor \theta^{-1}  \rfloor+1$;
            \2 $v=\varepsilon e_m + a e_{\lfloor \theta^{-1}  \rfloor+1}$ for some $m \geq \lfloor \theta^{-1}  \rfloor+2$, some $\varepsilon \in \{-1,1\}$ and some $|a|\leq \theta$,
    \end{outline}        
    and for $\alpha > 1$ holds 
    \begin{itemize}
    \item[(3)]  $\min ( \|u+y\|, \|u-y\| ) \leq 1$ for all $y \in \mathbb{S}_{T[\theta, \mathcal{S}_{\alpha}]}$ if and only if $u = \pm e_1$;
    \item[(4)]  If $v > 1$ and $\min ( \|v+y\|, \|v-y\| ) \leq \theta \cdot (\lfloor \theta^{-1}  \rfloor+1)$ for all $y \in \mathbb{S}_{T[\theta, \mathcal{S}_{\alpha}]}$, then $v= \pm e_m$ for some $m > 1$.
    \end{itemize}  
\end{lemma}

\begin{proof}
    \begin{outline}[enumerate] 
    \1 ($\Leftarrow$) 
    Proof of this implication is inspired by a part of the proof of \cite[Theorem 4.1]{antunes2022surjective}. Let $x = y \pm e_i$ for $i \leq \lfloor \theta^{-1} \rfloor$.         
    Suppose that the norm of $x$ is given by certain sets $d \leq E_1<E_2<\cdots < E_{d}$ for some $d\in \mathbb{N}$ in the sense that
\begin{equation}\label{bb}
    \|x\|= \theta \cdot \sum_{j=1}^{d}\big\|E_j x\big\|.
\end{equation}

It is enough to show that $\|x\| \leq \|y\|$. Suppose that $\min E_1 < i.$ Then
\begin{equation*}
   d \leq \min E_1 < i \leq {\theta^{-1}},
\end{equation*}
so
\begin{equation*}
    \theta \cdot \sum_{i=j}^{d}\big\|E_jx\big\|\leq \theta \cdot d \cdot \|x\| < \|x\|.
\end{equation*}
Hence \eqref{bb} cannot hold; a contradiction. Suppose that $\min E_1 = i.$ Then $d=i$, so from  \eqref{bb} we have $\|x\| = \|E_jx\|$ for any $1 \leq j \leq d$. Hence for $j \geq 2$ holds
\[
  \|x\| = \|E_jx\| = \|E_jy\| \leq \|y\|.   
\]
If $\min E_1 > i$ then obviously $\|x\| = \|y\|$.

Suppose that $\|x\| = \|x\|_{\infty} > 1$. Then $|y_i + 1| > 1$ and $|y_i - 1| > 1$; a contradiction.

     ($\Rightarrow$)
     The proof is similar to the proof of \cite{tan2012isometries}[Lemma 2.1.1]. Take any indices $\supp u < j_1 < j_2 < \ldots < j_{\lfloor \theta^{-1} \rfloor}$ instead of $2<k<l$.
     
     \1 Again, the proof is analogous to \cite{tan2012isometries}[Lemma 2.1.2]. Take $\varepsilon = \frac{\theta^{-1} - \lfloor \theta^{-1}\rfloor + 1 - \|v\|_{\infty}}{4} > 0$ at the beginning of the proof and indices $\max \{ m, i\} < j_1 < j_2 < \ldots < j_{\lfloor \theta^{-1}\rfloor}$ at the end of the proof (instead of indices $\max \{ m, i\} < n < n+1$).
    \end{outline}
    The proof of $(3)$ and $(4)$ is similar to the proofs of $(1)$ and $(2)$, respectively. Indeed, it is enough to take indices  $ j_1 < j_2 < \ldots < j_{\lfloor \theta^{-1}\rfloor} $ with additional assumption: $j_1 > {\lfloor \theta^{-1}\rfloor}$.
 \end{proof}   
 
 \begin{lemma}\label{l4}
       Let $x \in \mathbb{S}_{T[\theta, \mathcal{S}_{\alpha}]}$. Then $\|x+e_n\| = 2$ if and only if $x(n) = 1$. 
 \end{lemma}   

 \begin{proof}
     We omit the proof of implication ($\Leftarrow$) because it is trivial. Assume that $\|x+e_n\| = 2$. It is enough to show that norm of vector $x+e_n$ is the supremum norm. 
     
    Take any sets $d \leq E_1 < E_2<\cdots < E_d$. We may assume that $n \in E_{i_0}$ for some $i_0 \in \{1,2,\ldots,d\}$. Indeed, if this is not the case, we will not get a norm of vector $x+e_n$ greater than $1$, because $\|x\|=1$.

    Since $E_{i}e_n = 0$ for $i \neq i_0$ and $\|x\|=1$ we obtain
    \begin{equation*}
    \begin{split}
        \theta \sum_{i=1}^{d} \|E_i (x+e_n) \| 
        &\leq  \theta \sum_{i=1}^{d} ( \|E_i x \| + \|E_i e_n \| ) \\
        &= \theta \sum_{i=1}^{d}\|E_i x \| + \theta \|E_{i_0} e_n \| \\
        &\leq \theta ( \theta^{-1} + 1 ) < 2.
    \end{split}
    \end{equation*}
 \end{proof}

The proof of the subsequent lemma is analogous to the proof of \cite[Lemma 2.3]{tan2012isometries}, so we skip it.

\begin{lemma}\label{l5}
    If $U\colon \mathbb{S}_{T[\theta, \mathcal{S}_{\alpha}]} \to \mathbb{S}_{T[\theta, \mathcal{S}_{\alpha}]}$ is an isometry satisfying $-U( \mathbb{S}_{T[\theta, \mathcal{S}_{\alpha}]}) \subset U( \mathbb{S}_{T[\theta, \mathcal{S}_{\alpha}]})$, then $-U(e_i)=U(-e_i)$ for $i \in \mathbb{N}$.
\end{lemma}

\begin{lemma}\label{l7}
     Let $\theta^{-1} \geq 2$ be an integer and let $U\colon \mathbb{S}_{T[\theta, \mathcal{S}_{\alpha}]} \to \mathbb{S}_{T[\theta, \mathcal{S}_{\alpha}]}$ be surjective isometry.  If $\alpha = 1$ then 
     \[
        U e_i = \left\{\begin{array}{ll} \varepsilon_i e_{\pi(i)}, &  1\leq i\leq \theta^{-1} \\
        \varepsilon_i e_i, & i >  \theta^{-1} 
        \end{array} \right. \quad(i\in \mathbb N),
    \]
    and if  $\alpha > 1$ then $Ue_i = \varepsilon_i e_i$, where $(\varepsilon_i)_{i=1}^\infty$ is some $\{-1,1\}$-valued sequence and $\pi$ is a permutation of $\big\{1,2,\ldots,  \theta^{-1}  \big\}.$
\end{lemma}

\begin{proof}

\noindent\emph{Case 1}. Let $\alpha =1$.
 \item[] \emph{Step 1}. Fix $1\leq i \leq \theta^{-1}$.
 
    For any $y \in \mathbb{S}_{T[\theta, \mathcal{S}_1]}$ there exists $x \in \mathbb{S}_{T[\theta, \mathcal{S}_1]}$ such that $U(x)=y$. Since $U$ is isometry, so
    \begin{equation*}
        \|U(e_i) - y\| = \|U(e_i) - U(x)\| = \|e_i - x\|.
    \end{equation*}
    By Lemma \ref{l5} we obtain
     \begin{equation*}
        \|U(e_i) + y\| = \|-U(-e_i) + U(x)\| = \|e_i + x\|.
    \end{equation*}
    Hence by Lemma \ref{l3} (1)
     \begin{equation*}
        \min \{ \|U(e_i) + y\|, \|U(e_i) - y\| \} = \min \{ \|e_i + x\|, \|e_i - x\| \} \leq 1. 
    \end{equation*}
    Thus, again by Lemma \ref{l3} (1), for each  $i$ there is index $\pi(i) \in \big\{1, 2, \ldots ,{ \theta^{-1} }\big\}$ so that $U(e_i) = \pm e_{\pi(i)}$.  
    Note that 
    \begin{equation*}
        \qquad \quad \, 1= \| e_i \pm e_j \| = \big\|U(e_i) \pm U(e_j)\big\| = \|e_{\pi(i)} \pm e_{\pi(j)}\|   
    \end{equation*}
 for any $j\neq i$ in $\big\{1, 2, \ldots, { \theta^{-1}}\big\}.$ Therefore $\pi(j) \neq \pi(i)$ for  $j\neq i,$ so $\pi$ is the desired permutation. 

 \item[] \emph{Step 2}. Let $i >  \theta^{-1}$.  We will show that there is  $\varepsilon_i \in \{-1,1\}$ such that
    \begin{equation*}\label{e37}
    U(e_i) = \varepsilon_i e_{\sigma(i)},
  \end{equation*}
  for some permutation $\sigma$ of the set $\mathbb{N}\setminus\{1,2,\ldots, \theta^{-1} \}$.

Note that
\begin{equation*}
        \qquad \quad \, 1= \| e_i \pm e_j \| = \big\|U(e_i) \pm U(e_j)\big\| = \|U(e_i) \pm e_{\pi(j)}\|   
    \end{equation*}
 for any $j$ in $\big\{1, 2, \ldots, { \theta^{-1} }\big\},$ so $U(e_i) >  \theta^{-1}$. 
 Following the arguments in the proof of Lemma \ref{l4} it follows that for any $x \in \mathbb{S}_{T[\theta, \mathcal{S}_1]}$ holds
 \[
    \min \{ \|e_i - x\|, \|e_i + x\| \} \leq \theta + 1
 \]
 Since $U$ is surjective, so
 \[
    \min \{ \|U(e_i) - y\|, \|U(e_i) + y\| \} \leq \theta + 1
 \]
 for any $y \in \mathbb{S}_{T[\theta, \mathcal{S}_1]}$.
 
 By the Lemma \ref{l3} (b) there are $\sigma, \tilde{\sigma} \colon \mathbb{N}\setminus\{1,2,\ldots, \theta^{-1} \} \rightarrow \mathbb{N}\setminus\{1,2,\ldots, \theta^{-1} \}$ such that 
 \begin{equation}\label{e333}
    \big|\big(U(e_i)\big)(\sigma(i))\big| = 1 \quad \text{and} \quad \big|\big(U^{-1}(e_i)\big)(\tilde{\sigma}(i))\big| = 1,
  \end{equation}
 for all $i > \theta^{-1}$.
 We claim that for all $k,i> \theta^{-1}$ with $k \neq i$ we have 
 \begin{equation}\label{e34}
        \big(U(e_k)\big)(\sigma(i)) = 0 \quad \text{and} \quad \big(U^{-1}(e_k)\big)(\tilde{\sigma}(i)) = 0.
 \end{equation}
  Indeed, 
 \[
 1 = \|e_i \pm e_k\| = \|U(e_i) \pm U(e_k)\| \geq  \big|1 \pm \big(U(e_k)\big)(\sigma(i))\big|,
 \]
 for $k \neq i$ in $\mathbb{N}\setminus\{1,2,\ldots, \theta^{-1} \}$, and similarly for $U^{-1}$, so the conclusion follows. In particular $\sigma$ and $\tilde{\sigma}$ are injective.

  We will show that there exists $l > \theta^{-1}$ such that $\big|\big(U(e_l)\big)(\theta^{-1}+1)\big| = 1$. 
  
  If $\big|\big(U(e_{\theta^{-1}+1})\big)(\theta^{-1}+1)\big| = 1$ then the thesis is fulfilled, so suppose that this is not the case. Since $\theta^{-1}$ is an integer, so by Lemma \ref{l3} (b) we have $U(e_{\theta^{-1}+1}) = ae_{\theta^{-1}+1} + \varepsilon e_m$ for some $m > \theta^{-1}+1$, some $|a| \leq \theta$ and some $\varepsilon \in \{-1,1\}$. Then
 \begin{equation}\label{e33}
    1 > \| U(e_{\theta^{-1}+1}) - \varepsilon e_m \| = \| e_{\theta^{-1}+1} - \varepsilon \cdot U^{-1}(e_m) \|.
 \end{equation}
 Moreover, by \eqref{e333}, we have $|\big(U^{-1}(e_m)\big)(\tilde{\sigma}(m))|=1$. If $\tilde{\sigma}(m) > \theta^{-1}+1$ then 
 \[
 \| e_{\theta^{-1}+1} - \varepsilon \cdot U^{-1}(e_m) \| \geq \| e_{\theta^{-1}+1} - \varepsilon \cdot U^{-1}(e_m) \|_{\infty} \geq 1,
 \]
 so we obtain a contradiction with \eqref{e33}.
 
 This means that $\tilde{\sigma}(m)=\theta^{-1}+1$, \emph{i.e.}, $\big|\big(U^{-1}(e_m)\big)(\theta^{-1}+1)\big|=1$. 
 Hence from \eqref{e34} we have $\big(U^{-1}(e_{\theta^{-1}+1})\big)(\theta^{-1}+1) = 0$. 
 
 This together with Lemma \ref{l3} (b) yields $U^{-1}(e_{\theta^{-1}+1})=\tilde{\varepsilon} e_{\tilde{\sigma}(\theta^{-1}+1)}$ for some $\tilde{\varepsilon} \in \{-1,1\}$. 
 
 So $U(e_{\tilde{\sigma}(\theta^{-1}+1)}) = \tilde{\varepsilon} e_{\theta^{-1}+1}$, hence $\tilde{\sigma}(\theta^{-1}+1)$ is the $l$ we are looking for. 
 
 This together with \eqref{e34} gives us $\big(U(e_i)\big)(\theta^{-1}+1)=0$ for any $i \neq l$  with $i > \theta^{-1}$. By Lemma \ref{l3} (b) we obtain
 \begin{equation}\label{e35}
    U(e_i) = \varepsilon_i e_{\sigma(i)}
 \end{equation}
 for all $i \neq l$ with $i > \theta^{-1}$ and some $\{-1,1\}$-valued sequence $(\varepsilon_i)_{\substack{i=\theta^{-1}+1, \\ i \neq l}}^\infty$ Hence
 \[
    U^{-1}(e_{\sigma(i)}) = \varepsilon_i e_i
 \]
 for such $i$ and $(\varepsilon_i)_{\substack{i=\theta^{-1}+1, \\ i \neq l}}^\infty$ . This together with \eqref{e333} means that $\tilde{\sigma} = \sigma^{-1}$, so $\sigma$ is surjective. 
 Since for every $p \in \supp (U(e_l))\setminus \sigma(l)$ there exists $i \neq l$ with $\sigma(i) = p$, \emph{i.e.}, $\big|\big(U(e_i)\big)(p)\big| = 1$, by \eqref{e34} we have $\big(U(e_l)\big)(p)= 0$. Hence $U(e_l) = \varepsilon_l e_{\sigma(l)}$ for some $\varepsilon_l \in \{-1,1\}$. This together with  \eqref{e35} gives as conclusion. 

 \item[] \emph{Step 3}. We will show that $\sigma$ from \emph{Step 2} is an identity.
 
  Define 
\[
    x_k := k^{-1} \cdot \theta^{-1} \cdot (e_k, e_{k+1}, \ldots, e_{2k-1})
\]
for $k > \theta^{-1}$.
Then 
\[
    \|U(x_k) + U(e_i) \| = \| x_k + e_i \| = \|x_k - e_i \| = \| U(x_k) - U(e_i) \|
\]
for all $i \notin \supp x_k$. Hence, by the statement of \emph{Step 2} we have
\[
    \|U(x_k) + e_{\sigma(i)} \| = \| U(x_k) - e_{\sigma(i)}) \|
\]
for all $i \notin \supp x_k$.
Since $\|U(x_k)\|=1$ it must be $\big(U(x_k)\big)(\sigma(i)) = 0$ for all $i \notin \supp x_k$, so
\begin{equation}\label{nos}
    \supp U(x_k) \subseteq \{ \sigma(k), \sigma(k+1), \ldots, \sigma(2k-1) \}.
\end{equation}
We claim that $\sigma(k) \geq k$ for any $k>\theta^{-1}$.

Suppose that $\sigma(k) < k$. Then by \eqref{nos} there is $i \in \supp x_k$ such that
\[
    \big|\big(U(x_k)\big)(\sigma(i))\big| \geq (k-1)^{-1} \cdot \theta^{-1}. 
\]
Indeed, if not, then we obtain a contradiction, because
\[
    \| U(x_k) \| < \max \big\{ (k-1)^{-1} \cdot \theta^{-1},\, \theta \cdot (k-1) \cdot (k-1)^{-1} \cdot \theta^{-1} \big\} = 1.
\]
 Since $\sigma(\theta^{-1}+1) \geq \theta^{-1}+1$, assume firstly that $k \in \{\theta^{-1}+2, \theta^{-1}+3, \ldots, 2\theta^{-1}\}$. 
Then
\begin{equation*}
    \begin{split}
        1 + (k-1)^{-1} \cdot \theta^{-1} &\leq \big\| U(x_k) + \sgn \big( \big( U(x_k) \big) (\sigma(i)) \big) e_{\sigma(i)} \big\| \\
        &= \big\| x_k + \sgn \big( \big( U(x_k) \big) (\sigma(i)) \big) U^{-1}(e_{\sigma(i)}) \big\| \\
        &\leq \| x_k + e_{i} \| = \max \big\{ k^{-1} \cdot \theta^{-1} + 1, \theta \cdot \big(k^{-1} \cdot \theta^{-1} \cdot k + 1\big) \big\} \\
        &= 1 + k^{-1} \cdot \theta^{-1},
    \end{split}
\end{equation*}
which cannot hold. Hence for $\theta^{-1} < k \leq 2\theta^{-1}$ we have $\sigma(k) \geq k$. 

Assume that $k > 2\theta^{-1}$.  Since $x_k - e_i$ is a vector that has $k-1$ coordiantes equal to $k^{-1} \theta^{-1}$, one coordiante whose modulus is equal to $1-k^{-1} \theta^{-1}$ and the other coordinates equal to zero, we obtain
\begin{equation}\label{x}
    \begin{split}
    \big\| U(x_k) - \sgn \big( \big( U(x_k) \big) (\sigma(i)) \big) e_{\sigma(i)} \big\| &\geq \| x_k - e_i \| \\
    &= \max \big\{ 1- k^{-1} \cdot \theta^{-1},\, \theta \big( k^{-1} \cdot \theta^{-1}(k-2) +1 \big) \big\} \\
    &= 1 - 2k^{-1} + \theta > 1 = \| U(x_k) \|.
   \end{split}
\end{equation}
This means that $1-\big|\big( U(x_k) \big) (\sigma(i))\big| > \big|\big( U(x_k) \big) (\sigma(i))\big|$, so $1 - 2\big|\big( U(x_k) \big) (\sigma(i))\big|>0$.

For any finite set $E_j \subset \mathbb{N}$, where $j \in \mathbb{N}$ we have 
\begin{equation*}
    \begin{split}
       \big\| E_j \big( U(x_k) - \sgn \big( \big( U(x_k) \big) (\sigma(i)) \big) &e_{\sigma(i)} \big )\big\|\\ &\leq \big\| E_j \big( U(x_k) - 2 \big( U(x_k) \big) (\sigma(i)) e_{\sigma(i)} \big) \big\|\\ 
       &+ \big\| E_j \big( 2 \big( U(x_k) \big) (\sigma(i)) e_{\sigma(i)} - \sgn \big( \big( U(x_k) \big) (\sigma(i)) \big) e_{\sigma(i)} \big) \big\| \\
       &= \|E_jU(x_k)\|  \\
       &+ \big\| E_j \big( 2 \big( U(x_k) \big) (\sigma(i)) e_{\sigma(i)} - \sgn \big( \big( U(x_k) \big) (\sigma(i)) \big) e_{\sigma(i)} \big) \big\|, 
    \end{split}
\end{equation*}
The above equality holds because the corresponding vectors have equal coordinate moduli.
Multiplying by $\theta$ both sides of the above inequality and taking the supremum over all consecutive sets $d<E_1 < E_2 < \dots < E_{d}$ for some $d \in \mathbb{N}$, we obtain
\begin{equation*}
    \begin{split}
       \big\| \big( U(x_k) - \sgn \big( \big( U(x_k) \big) (\sigma(i)) \big)& e_{\sigma(i)} \big\|\\
       &\leq \|U(x_k)\| \\ &+ \theta \cdot \big\| E_{j_0} \big( 2 \big( U(x_k) \big) (\sigma(i)) e_{\sigma(i)} - \sgn \big( \big( U(x_k) \big) (\sigma(i)) \big) e_{\sigma(i)} \big) \big\| 
    \end{split}
\end{equation*}
for some $j_0 \in \{1,2,\ldots, d\}$. Hence
\begin{equation*}
    \begin{split}
       \big\| U(x_k) - \sgn \big( \big( U(x_k) \big) (\sigma(i)) \big) e_{\sigma(i)}\big\| &\leq \|U(x_k)\| + \theta\big(1- 2 \big|\big( U(x_k) \big) (\sigma(i))\big|\big) \\
       &= 1+ \theta - 2 \theta \big|\big( U(x_k) \big) (\sigma(i))\big| \\
       &\leq 1+ \theta  - 2 \theta (k-1)^{-1} \cdot \theta^{-1}.
    \end{split}
\end{equation*}
which contradicts \eqref{x}.

Doing the same for $U^{-1}$ instead of $U$ we obtain $\sigma^{-1}(\sigma(k)) \geq \sigma(k)$, so $k \geq \sigma (k)$, hence $\sigma(k) = k$ for all $k>\theta^{-1}$. This ends the proof for $\alpha = 1$.

\item[] \emph{Case 2}. Suppose that $\alpha= \beta +1$ for some $\beta < \omega_1$.

The proof that $U(e_1) = \varepsilon_1 e_1$, where $\varepsilon_1 \in  \{-1,1\}$ and for any $i > 1$ there is $\varepsilon_i \in \{-1,1\}$ such that
    \begin{equation*}\label{e377}
    U(e_i) = \varepsilon_i e_{\sigma(i)}
  \end{equation*}
  for some permutation $\sigma$ of set $\{2,3,\ldots\}$ is similar to the previous case and much simpler, so we omit it. We will show that $\sigma$ is an identity.

Fix $k > 1$ and suppose that $t:= \sigma(k) < k$.

Note that every $\mathcal S_{\alpha}$-set whose minimum is $k$ is the union of at most $k$ many $\mathcal S_{\beta}$-sets, so the idea of the proof of this case is to choose the indices $ j_1<j_2<\cdots<j_m,$ for some $m \in \mathbb{N},$ so that they creates $k$ many consecutive $\mathcal S_{\beta}$-sets. At the same time, we must ensure that the set \[
\{ \sigma(j_1), \sigma(j_2), \ldots, \sigma(j_m)  \} 
\]
associated with these indices is not $\mathcal S_{\alpha}$-set. We proceed as follows.
Choose indices $j_1 = k, j_1^1 > \max \big\{ k, \theta^{-1}\big\}$. Then take the next index $j_2^1 > \max \big\{ j_1^1, \sigma(j_1^1)\big\}$,
in sequence $j_3^1 > \max \big\{ j_2^1, \sigma(j_2^1)\big\}$ and so on. 

Following this procedure, we may choose $S^1_{\beta} = \{ j_1^1, j_2^1, \ldots, j_{m_1}^1 \}$ which is a~maximal $\mathcal S_{\beta}$-set. At the same time, we get the indices 
 \[
\{ \sigma(j_1^1), \sigma(j_2^1), \ldots, \sigma(j_{m_1}^1)  \} 
\]
so that 
\begin{equation*}
    \begin{split}
         \sigma(j_1^1) &< j_2^1 \leq \max \{ \, j_2^1, \sigma(j_2^1) \, \} < j_3^1
         \leq \cdots \\& \cdots \leq \max \{ \, j_{m_1-1}^1, \sigma(j_{m_1-1}^1) \, \} < j_{m_1}^1.
    \end{split}
\end{equation*}

We choose the remaining maximal $\mathcal S_{\beta}$-sets according to the following procedure. If $S^{n}_{\beta} = \{ j^{n}_1, j^{n}_2, \ldots, j_{m_n}^n \}$ for some $n<t$, then choose $S^{n+1}_{\beta} = \{ j^{n+1}_1, j^{n+1}_2, \ldots, j_{m_{n+1}}^{n+1} \}$, where $j_1^{n+1} > \max \{\, j_{m_n}^n, \sigma(j_{m_n}^n) \,\}$ and $j_p^{n+1} > \max \{\, j_{p-1}^{n+1}, \sigma(j_{p-1}^{n+1}) \,\}$ for $p \leq m_{n+1}$.

We finally arrive at indices 
\[
j_1^1<j_2^1 < \cdots <j_{m_t}^t,
\]
that form a union of $t$ maximal $\mathcal S_{\beta}$-sets, so we got the conclusion because we may choose $\mathcal S_{\beta}$-sets

\begin{equation}\label{sets}
        k \leq S^0_{\beta} < S^1_{\beta} < \dots < S^{t}_{\beta},
\end{equation}  
    where
    \begin{itemize}
        \item  
    $
    S^0_{\beta} = \{ j_1 \},
    $
        \item  
    $
    S^1_{\beta} = \{ j_1^1, j_2^1, \ldots, j_{m_1}^1 \},
    $
        \item
    $
    \vdots
    $  
        \item  
    $
    S^{t}_{\beta} = \{ j^{t}_1, j^{t}_2, \ldots, j_{m_t}^t \}.
    $
    \end{itemize}

    By the above construction, 
     \[
    \tilde{S}_m := \big\{\sigma(j_1),\sigma(j_1^1),\sigma(j_2^1), \ldots \sigma(j_{m_t}^t) \big\},
    \]
    where $m$ is the cardinality of set $\tilde{S}_m$, is not $\mathcal S_{\alpha}$-set because the Schreier family (of order $\beta$) is spreading (see Definition \ref{sp}). Indeed, suppose $\tilde{S}_m \in \mathcal{S}_{\alpha}$, then $\tilde{S}_m$ is the union of at most t-many successive $\mathcal S_{\beta}$-sets. Take any index $j > \max \{ j_{m_t}^t, \sigma(j_{m_t}^t) \}$. Then $j_1^1<j_2^1 < \cdots <j_{m_{t}}^t < j$ is the union of at most t-many successive $\mathcal S_{\beta}$-sets by the spreading property of $S_{\beta}$. This contradicts the choice of $j_1^1<j_2^1 < \cdots <j_{m_{t}}^t$ as the
    union of t-many maximal $S_{\beta}$-sets.

Then we define
\[
    x_k := \theta^{-1} \cdot m^{-1} \cdot \sum_{i =1}^m e_{j_i}.
\]
As in \eqref{nos} we have
\begin{equation*}
    \supp U(x_k) \subseteq \tilde{S}_m.
\end{equation*}
To complete the proof, it is enough to replace each $k$ with $m$ in \emph{Step 3} of \emph{Case 1}. Note that 
\[
    \big|\big(U(x_k)\big)(\sigma(i))\big| \geq (m-1)^{-1} \cdot \theta^{-1} 
\]
holds for some $i \in \{j_1,j_2, \ldots, j_m\}$ as we ensured that $m > \theta^{-1} +1$.

\item[] \emph{Case 3}: Suppose that $\alpha$ is a limit ordinal.

We proceed as in Case $2$ for $\alpha= \alpha_{t},$ where $(\alpha_i)_{i=1}^{\infty}$ is a fixed strictly increasing sequence of successor ordinals converging to $\alpha$ with  $\mathcal S_{\beta_i} \subset \mathcal S_{\beta_n}$ for $i \leq n$, where $\alpha_n := \beta_{n} +1$ for each $n \in \mathbb{N}$, choosing suitable sequence $(j_i)_{i=1}^m$.
        Indeed, $\mathcal S_{\beta_t}$-sets $k\leq S^0_{\beta_{t}} < S^1_{\beta_t} < \dots < S^{t}_{\beta_{t}}$, where
        \begin{itemize}
        \item  
    $
    S^0_{\beta_{t}} = \{ j_1 \},
    $
        \item  
    $
    S^1_{\beta_{t}} = \{ j_1^1, j_2^1, \ldots, j_{m_1}^1 \},
    $
        \item
    $
    \vdots
    $  
        \item  
    $
    S^{t}_{\beta_{t}} = \{ j^{t}_1, j^{t}_2, \ldots, j_{m_t}^t \}.
    $
    \end{itemize}
    give rise to an $\mathcal S_{\alpha}$-set (even an $\mathcal S_{\alpha_{t}}$-set). Moreover, the set 
      \[
    \tilde{S}_m := \big\{\sigma(j_1),\sigma(j_1^1),\sigma(j_2^1), \ldots \sigma(j_{m_t}^t) \big\}
    \]
    is not $\mathcal S_{\alpha_{t}}$-set by the spreading property of $S_{\beta_n}$ (as explained in Case 2). Hence $\tilde{S}_m \notin \mathcal S_{\alpha}$ as we ensured that $\mathcal S_{\beta_i} \subset \mathcal S_{\beta_n}$ for $i \leq n$. Indeed, suppose $\tilde{S}_m \in \mathcal S_{\alpha}$. Then $\tilde{S}_m \in \mathcal S_{\alpha_j}$ for some $j \leq t$, i.e. $\tilde{S}_m$ is the union of at most $j$-many successive $\mathcal S_{\beta_{j}}$-sets, i.e. $\mathcal S_{\beta_t}$-sets by the assumption on $(\beta_i)_{i}$. This means that $\tilde{S}_m \in \mathcal S_{{\alpha}_t}$; a contradiction.
\end{proof}

\begin{lemma}\label{l8}
    If $y = \sum_{i=1}^{\infty} b_i e_i \in \mathbb{S}_{T[\theta, \mathcal{S}_{\alpha}]}$, then $x = \sum_{i=1, i \neq j}^{\infty}  \theta b_i e_i \pm e_j \in \mathbb{S}_{T[\theta, \mathcal{S}_{\alpha}]}$. Moreover the norm of vector $z = y - \sum_{i=1, i \neq j}^{\infty}  \theta b_i e_i - \sgn(b_j) e_j$ is equal to $1+|b_j|$.
\end{lemma}
\begin{proof} 
        Let us note that
        \[
        x = (\theta b_1, \theta b_2, \ldots, \theta b_{j-1}, \pm 1, \theta b_{j+1} \ldots).
        \]
        The supremum norm is obviously not greater than 1, because $\|y\|=1$.
        Take any sets $d \leq E_1 < E_2<\cdots < E_d$. We may assume that $j \in E_{i_0}$ for some $i_0 \in \{1,2,\ldots,d\}$. Indeed, if this is not the case, we will not get a norm of vector $x$ greater than $1$, because $\|y\|=1$.

    Since 
    \begin{equation*}
    \begin{split}
        \theta \sum_{i=1}^{d} \|E_i x \| 
        &\leq  \theta \sum_{i=1}^{d} ( \|E_i (\theta y) \| + \|E_i e_j \| ) \\
        &= \theta \sum_{i=1}^{d} \theta \|E_i y \| + \theta \|E_{i_0} e_j \| \\
        &\leq \theta ( \theta \cdot \theta^{-1} + 1 ) = 2\theta \leq 1,
    \end{split}
    \end{equation*} 
    so $\|x\| \leq 1$ as desired.
    The vector $z$ is of the form
     \[
        z = \big( (1 - \theta) b_1,\, (1 - \theta) b_2, \ldots, (1 - \theta) b_{j-1},\, \pm (1 + |b_j|),\, (1 - \theta) b_{j+1}, \ldots \big).
     \]
     The supremum norm is $1 + |b_j|$. As before, we show that the norm given by some sets $d \leq E_1 < E_2<\cdots < E_d$ is not larger. Indeed,
      \begin{equation*}
    \begin{split}
        \theta \sum_{i=1}^{d} \|E_i z \| 
        &\leq  \theta \sum_{i=1}^{d} \big( (1 - \theta) \|E_i  y \| + (1 + |b_j|) \|E_i e_j \| \big) \\
        &\leq \theta \big( (1 - \theta) \theta^{-1} + 1 + |b_j| \big) \\
        &= 1 + \theta |b_j| < 1 + |b_j|.
    \end{split}
    \end{equation*}    
\end{proof}

We are now ready to prove Theorem~\ref{Th:B}.
\begin{proof}
    Fix $\alpha=1$. Let $\theta^{-1} \geq 2$ be an integer. By Lemma \ref{l7} isometry $U$ is of the form
     \[
        U e_i = \left\{\begin{array}{ll} \hat{\varepsilon}_i e_{\pi(i)}, &  1\leq i\leq \theta^{-1} \\
        \hat{\varepsilon}_i e_i, & i >  \theta^{-1} 
        \end{array} \right. \quad(i\in \mathbb N),
    \]
    where $(\hat{\varepsilon}_i)_{i=1}^\infty$ is some $\{-1,1\}$-valued sequence and $\pi$ is a permutation of $\big\{1,2,\ldots, \theta^{-1}  \big\}.$ 
    
    Define $\hat{\pi}(i)$ as $\pi(i)$ for $1\leq i\leq \theta^{-1}  $ and $\hat{\pi}(i) = i$ for $i > \theta^{-1} $.  For $i \in \mathbb{N}$ let us set $\varepsilon_i := \big(U(e_i)\big)\big({\hat{\pi}(i)}\big)$.
    
    Fix $x = \sum_{i=1}^{\infty} a_i e_i \in \mathbb{S}_{T[\theta, \mathcal{S}_1]}$  and take $y = \sum_{i=1}^{\infty} b_i e_i \in \mathbb{S}_{T[\theta, \mathcal{S}_1]}$ such that $U(x)=y$. 
    If $a_i$ is nonzero and $b_i = 0$ then we use the convention that $\sgn(b_i) = 1$.
    Fix $j \in \mathbb{N}$ and take
    \[
    y_j = \sum_{i=1, i \neq \hat{\pi}(j)}^{\infty} \theta b_i e_i - \varepsilon_j \sgn(a_j) e_{\hat{\pi}(j)}
    \]
    and 
    \[
    z_j = \sum_{i=1, i \neq \hat{\pi}(j)}^{\infty} \theta b_i e_i - \sgn(b_{\hat{\pi}(j)}) e_{\hat{\pi}(j)}.
    \]
    By Lemma \ref{l8} we have $\|y_j\|=1$ and $\|y-z_j\|=1+|b_{\hat{\pi}(j)}|$. 
    Let $x_j \in \mathbb{S}_{T[\theta, \mathcal{S}_1]}$ be such that $U(x_j) = y_j$. We obtain 
    \[
    \|x_j - \sgn(a_j) e_j \| = \|U(x_j) - \sgn(a_j) U(e_j) \| = \|y_j - \sgn(a_j) \varepsilon_j e_{\hat{\pi}(j)} \| = 2.
    \]
    So by Lemma \ref{l4} we have $x_j(j) = - \sgn(a_j)$.
    This yields 
    \[
    1 + |b_{\hat{\pi}(j)}| = \| y - z_j \| \geq \| y - y_j \| = \| U(x) - U(x_j) \| = \| x - x_j \| \geq 1 + |a_j|.
    \]
    Hence 
    \begin{equation}\label{e1}
        |b_{\hat{\pi}(j)}| \geq |a_j|.
    \end{equation}
    Note that $\varepsilon_i = \big( U^{-1}(e_{\hat{\pi}(i)})\big)(i)$ and $U^{-1} (e_i) = \varepsilon_{\hat{\pi}^{-1}(i)} e_{\hat{\pi}^{-1}(i)}$. Similarly, we define
     \[
    u_j = \sum_{i=1, i \neq \hat{\pi}^{-1}(j)}^{\infty} \theta a_i e_i - \varepsilon_{\hat{\pi}^{-1}(j)} \sgn(b_j) e_{\hat{\pi}^{-1}(j)}
    \]
    and 
    \[
    v_j = \sum_{i=1, i \neq \hat{\pi}^{-1}(j)}^{\infty} \theta a_i e_i - \sgn(a_{\hat{\pi}^{-1}(j)}) e_{\hat{\pi}^{-1}(j)}.
    \]
    Then
    \[
    \|U(u_j) - \sgn (b_j) e_j \| = \| u_j - \sgn(b_j) \varepsilon_{\hat{\pi}^{-1}(j)} e_{\hat{\pi}^{-1}(j)} \| = 2.
    \]
    So $\big(U(u_j)\big)(j) = - \sgn(b_j)$.
    Hence 
    \[
    1 + |a_{\hat{\pi}^{-1}(j)}| = \| x - v_j \| \geq \| x - u_j \| = \| U(x) - U(u_j) \| = \| y - U(u_j) \| \geq 1 + |b_j|.
    \]
    This means that $|a_{\hat{\pi}^{-1}(j)}| \geq |b_j|$, which together with \eqref{e1} gives us $|a_{\hat{\pi}^{-1}(j)}| = |b_j|$. We moreover have $\| x - v_j \| = \| x - u_j \|$, so $\varepsilon_{\hat{\pi}^{-1}(j)} \sgn(b_j) = \sgn(a_{\hat{\pi}^{-1}(j)})$ and finally $b_{\hat{\pi}(j)} = \varepsilon_j a_j $ for $j \in \mathbb{N}$, hence the conclusion follows.

     Fix  $1<\alpha< \omega_1$ and let $Ue_i = \hat{\varepsilon}_i e_i$, where $(\hat{\varepsilon}_i)_{i=1}^\infty$ is some $\{-1,1\}$-valued sequence. The proof is exactly the same if we define $\hat{\pi}(i)$ as identity for any $i \in \mathbb{N}$.
\end{proof}

\bibliography{bibliography.bib}

\begin{thebibliography}{10}

\bibitem{alspach1992complexity}
Dale~E. Alspach and Spiros Argyros.
\newblock Complexity of weakly null sequences.
\newblock {\em Dissertationes Math. (Rozprawy Mat.)}, 321:44, 1992.

\bibitem{antunes2022surjective}
L.~Antunes and K.~Beanland.
\newblock Surjective isometries on {Banach} sequence spaces: A survey.
\newblock {\em Concrete Operators}, 9(1):19--40, 2022.

\bibitem{todorcevic2005ramsey}
Spiros~A. Argyros and Stevo Todor\v{c}evi\'c.
\newblock {\em Ramsey Methods in Analysis}.
\newblock Adv. Courses in Math. -- CRM Barc. Basel: Birkh\"auser, 2005.

\bibitem{banakh2022every}
Taras Banakh.
\newblock Every 2-dimensional banach space has the {M}azur--{U}lam property.
\newblock {\em Linear Algebra and its Applications}, 632:268--280, 2022.

\bibitem{causey2017concerning}
Ryan~M Causey.
\newblock Concerning the szlenk index.
\newblock {\em Studia Mathematica}, 236:201--244, 2017.

\bibitem{cueto2022exploring}
Mar{\'\i}a Cueto-Avellaneda, Daisuke Hirota, Takeshi Miura, and Antonio~M
  Peralta.
\newblock Exploring new solutions to {T}ingley’s problem for function
  algebras.
\newblock {\em Quaestiones Mathematicae}, pages 1--32, 2022.

\bibitem{ding2004representation}
Guang~Gui Ding.
\newblock The representation theorem of onto isometric mappings between two
  unit spheres of {{$\ell^1(\Gamma)$}}-type spaces and the application to the
  isometric extension problem.
\newblock {\em Acta Mathematica Sinica}, 20(6):1089--1094, 2004.

\bibitem{ding2007isometric}
Guang-Gui Ding.
\newblock The isometric extension of the into mapping from a
  {{$L^{\infty}(\Gamma)$}}-type space to some {B}anach space.
\newblock {\em Illinois Journal of Mathematics}, 51(2):445--453, 2007.

\bibitem{ding20021}
Guanggui Ding.
\newblock The 1-{L}ipschitz mapping between the unit spheres of two {H}ilbert
  spaces can be extended to a real linear isometry of the whole space.
\newblock {\em Science in China Series A: Mathematics}, 45:479--483, 2002.

\bibitem{ding2003isometric}
Guanggui Ding.
\newblock The isometric extension problem in the unit spheres of
  {{$\ell^p(\Gamma)$}} {{$(p > 1)$}} type spaces.
\newblock {\em Science in China Series A: Mathematics}, 46:333--338, 2003.

\bibitem{dding2004representation}
Guanggui Ding.
\newblock The representation theorem of onto isometric mappings between two
  unit spheres of {{$\ell^{\infty}$}}-type spaces and the application on
  isometric extension problem.
\newblock {\em Science in China Series A: Mathematics}, 47:722--729, 2004.

\bibitem{ding2008isometric}
GuangGui Ding.
\newblock The isometric extension of “into” mappings on unit spheres of
  {AL}-spaces.
\newblock {\em Science in China Series A: Mathematics}, 51(10):1904--1918,
  2008.

\bibitem{ding2009isometric}
GuangGui Ding.
\newblock On isometric extension problem between two unit spheres.
\newblock {\em Science in China Series A: Mathematics}, 52(10):2069--2083,
  2009.

\bibitem{fang2006extension}
Xi~Nian Fang and Jian~Hua Wang.
\newblock Extension of isometries between the unit spheres of normed space {E}
  and {{$C(\Omega)$}}.
\newblock {\em Acta Mathematica Sinica. English Series}, 22(6):1819, 2006.

\bibitem{jian2004extension}
Wang Jian.
\newblock On extension of isometries between unit spheres of {{$AL_p$}}- spaces
  {{$(0<p<\infty)$}}.
\newblock {\em Proceedings of the American Mathematical Society},
  132(10):2899--2909, 2004.

\bibitem{liu2007extension}
Rui Liu.
\newblock On extension of isometries between unit spheres of
  {{$L^{\infty}(\Gamma)$}}-type space and a banach space {E}.
\newblock {\em Journal of mathematical analysis and applications},
  333(2):959--970, 2007.

\bibitem{liu2009extension}
Rui Liu and Lun Zhang.
\newblock On extension of isometries and approximate isometries between unit
  spheres.
\newblock {\em Journal of mathematical analysis and applications},
  352(2):749--761, 2009.

\bibitem{maslany2023isometries}
Natalia Ma{\'s}lany.
\newblock {{I}sometries of combinatorial {T}sirelson spaces}.
\newblock {\em Proc. Amer. Math. Soc.}, 151:4475--4484, 2023.

\bibitem{peralta2018survey}
Antonio~M Peralta.
\newblock A survey on tingley’s problem for operator algebras.
\newblock {\em Acta Scientiarum Mathematicarum}, 84:81--123, 2018.

\bibitem{tan2009nonexpansive}
Dong-Ni Tan.
\newblock Nonexpansive mappings on the unit spheres of some banach spaces.
\newblock {\em Bulletin of the Australian Mathematical Society},
  80(1):139--146, 2009.

\bibitem{tan2012extension}
Dong~Ni Tan.
\newblock Extension of isometries on the unit sphere of $l^p$ spaces.
\newblock {\em Acta Mathematica Sinica, English Series}, 28:1197--1208, 2012.

\bibitem{tan2012isometries}
Dong-Ni Tan.
\newblock Isometries of the unit spheres of the {T}sirelson space {$T$} and the
  modified {T}sirelson space {$T_M$}.
\newblock {\em Houston J. Math}, 2012.

\bibitem{tingley1987isometries}
Daryl Tingley.
\newblock Isometries of the unit sphere.
\newblock {\em Geometriae Dedicata}, 22(3):371--378, 1987.

\bibitem{tsirelson1974impossible}
B.~Tsirelson.
\newblock It is impossible to embed $\ell_p$ or $c_0$ into an arbitrary
  {Banach} space ({Russian}).
\newblock {\em {Funkts. Anal. i Prilozhen} English translation: {Funct. Anal.
  Appl}}, 8:138--141, 1974.

\bibitem{yang2006extension}
Xiuzhong Yang.
\newblock On extension of isometries between unit spheres of {{$L_p(\mu)$}} and
  {{$L_p(\nu, H)$ $(1< p\neq 2)$}}, {H} is a {H}ilbert space).
\newblock {\em Journal of mathematical analysis and applications},
  323(2):985--992, 2006.

\bibitem{yang2014extension}
Xiuzhong Yang and Xiaopeng Zhao.
\newblock On the extension problems of isometric and nonexpansive mappings.
\newblock In {\em Mathematics Without Boundaries: Surveys in Pure Mathematics},
  pages 725--748. Springer, 2014.

\end{thebibliography}
\bibliographystyle{plain}

\end{document}